\documentclass[11pt,a4paper]{amsart}

\author{V. Beresnevich}
\address[V. Beresnevich]{Department of Mathematics, University of York, Heslington, York, YO10
5DD, United Kingdom}
\email{victor.beresnevich@york.ac.uk}
\author{J. Levesley}
\address[J. Levesley]{Department of Mathematics, University of York, Heslington, York, YO10
5DD, United Kingdom}
\email{jason.levesley@york.ac.uk}
\author{B. Ward}
\address[B. Ward]{Department of Mathematics, University of York, Heslington, York, YO10
5DD, United Kingdom}
\email{bw744@york.ac.uk}
\usepackage[utf8]{inputenc}
\usepackage{amsmath}
\usepackage{amsfonts}
\usepackage{amssymb}
\newtheorem{theorem}{Theorem}

\newtheorem{lemma}[theorem]{Lemma}
\newtheorem{corollary}[theorem]{Corollary}
\newtheorem{remark}[theorem]{Remark}

\newcommand{\ha}{\mathcal{H}}
\newcommand{ \s}{\mathcal{S}}
\newcommand{\Q}{\mathbb{Q}}
\newcommand{\U}{\mathcal{U}}
\newcommand{\R}{\mathbb{R}}
\newcommand{\N}{\mathbb{N}}
\newcommand{\Z}{\mathbb{Z}}
\newcommand{\bx}{\boldsymbol{x}}
\newcommand{\bt}{\boldsymbol{\tau}}
\newcommand{\tb}{\tilde{\tau}}
\newcommand{\M}{\mathcal{M}}
\newcommand{\q}{\mathcal{Q}}
\newcommand{\bp}{\boldsymbol{p}}
\newcommand{\ba}{\boldsymbol{a}}

\newcommand{\m}{\mathcal{M}}
\newcommand{\w}{\mathcal{W}}

\newcommand{\by}{\mathbf{y}}

\makeindex
\title[Weighted Approximable points over Manifolds]{A Lower Bound for the Hausdorff Dimension of the set of Weighted Simultaneously Approximable points over Manifolds}
\usepackage{enumitem}
\usepackage{subcaption}
\usepackage{amsthm}
\usepackage{bookmark}
\newtheorem*{theorem*}{Theorem}
\begin{document}

\date{October, 2020}
\begin{abstract}
Given a weight vector $\bt=(\tau_{1}, \dots, \tau_{n}) \in \R^{n}_{+}$ with each $\tau_{i}$ bounded by certain constraints, we obtain a lower bound for the Hausdorff dimension of the set $\w_{n}(\bt) \cap \m$, where $\m$ is a twice continuously differentiable manifold. From this we produce a lower bound for $\w_{n}(\Psi) \cap \m$ where $\Psi$ is a general approximation function with certain limits. The proof is based on a technique developed by Beresnevich et al. in \cite{BLVV17}, but we use an alternative mass transference style theorem proven by Wang, Wu and Xu \cite{WWX15} to obtain our lower bound.
\end{abstract}
\maketitle

\section{Introduction}

Diophantine approximation over $\mathbb{R}$ is largely complete with regards to the Lebesgue measure and Hausdorff measure when considering monotonic approximation functions. In higher dimensions there are two branches of Diophantine approximation studied in depth; dual and simultaneous. In the simultaneous setting we consider the set of numbers
\begin{equation*}
\s_{n}( \psi ) := \left\{ \bx \in \mathbb{R}^{n} : \max_{1 \leq i \leq n} |qx_{i}-p_{i}|< \psi(q) \text{   for i.m. } (\bp,q) \in \mathbb{Z}^{n} \times \mathbb{N} \right\},
\end{equation*}
where $\psi : \mathbb{N} \rightarrow \mathbb{R}_{+}$ for $\R_{+}=\{a \in \R: a \geq 0\}$ is an approximation function, with $\psi(q) \to 0$ as $q \to \infty$. Many measure results for this set have already been established. For example, where $\psi$ is of the form $\psi(q)=q^{-\tau}$ with $\tau  \leq 1/n$ then Dirichlet's theorem states that there are infinitely many $(\bp,q) \in \mathbb{Z}^{n} \times \mathbb{N}$ for any $\bx \in \mathbb{R}^{n}$ i.e. $\s_{n}( \tau) = \mathbb{R}^{n}$. Further, where $\tau > 1/n$ we can deduce via the Borel-Cantelli Lemma \cite{B09}, \cite{C17} that such set is of Lebesgue measure zero. A useful definition when stating metric results is the notion of \textit{extremality} \cite{S79}. For a point $\bx \in \R^{n}$, if the exponent of the approximation function $q^{-1/n}$ is the largest possible, i.e. for any $\tau>1/n$ we have that $\bx \not \in \s_{n}(\tau)$, then $\bx$ is said to be extremal. In the case that there does exist $\tau>1/n$ such that $\bx \in \s_{n}(\tau)$, then we say $\bx$ is very well approximable. \par 

 Using Hausdorff measure and Hausdorff dimension we can be more precise in distinguishing between two sets of zero Lebesgue measure.
For a set $X \subset \R^{n}$ and $\rho>0$, define a $\rho$-cover of $X$ as a sequence of balls $\{B_{i}\}$ such that $X \subset \bigcup_{i}B_{i}$, and for all balls $r(B_{i}) \leq \rho$, where $r(.)$ is the radius. Define a dimension function $f:\R_{+} \rightarrow \R_{+}$  as a continuous increasing function with $f(r) \rightarrow 0$ as $r \rightarrow 0$. Let
\begin{equation*}
\ha_{\rho}^{f}(X)=\inf \left\{ \sum_{i} f(r(B_{i})) : \{B_{i}\} \, \text{ is a } \, \rho-\text{cover of } X \right\},
\end{equation*}
where the infimum is take over all $\rho$-covers of $X$. Then define the $f$-Hausdorff measure by 
\begin{equation*}
\ha^{f}(X)= \lim_{\rho \rightarrow 0} \ha^{f}_{\rho}(X).
\end{equation*}
When the dimension function $f(x)=x^{s}$ for some $s \in \R_{+}$, we will denote $\ha^{f}$ as $\ha^{s}$. With this notation define the Hausdorff dimension as 
\begin{equation*}
\dim X = \inf \left\{ s>0: \ha^{s}(X)< \infty \right\} = \sup \left\{ s>0 : \ha^{s}(X)= \infty \right\},
\end{equation*}
where the latter is true if $X$ is an infinite set. The following theorem is a foundational result, proved independently by Jarnik \cite{J29} and Besicovitch \cite{B34}, on the Hausdorff dimension of the set of simultaneously approximable points.
\begin{theorem*} \label{jb}
Let $\tau \geq \frac{1}{n}$ then, 
\begin{equation*}
\dim \s_{n}(\tau) = \frac{n+1}{\tau+1}.
\end{equation*}
\end{theorem*}
The result allows us to differentiate between the size of simultaneous $\tau$-approximable sets where $\tau \geq 1/n$. \par
 This form of simultaneous approximation can be thought of as the set of all $\bx \in \mathbb{R}^{n}$ that lie in infinitely many $n$-dimensional hypercubes with side length $2\psi(q)$ and centres $\mathbf{p}/q=(\frac{p_{1}}{q}, \dots, \frac{p_{n}}{q}) \in \mathbb{Q}^{n}$, where each $p_{i} \in \mathbb{Z}$. Weighted simultaneously approximable numbers can be considered in a very similar way, with the $n$-dimensional hypercubes being replaced by $n$-dimensional hyperrectangles. This set is defined as 
 \begin{equation*}
  \w_{n}(\Psi):= \left\{ \bx \in \mathbb{R}^{n} : \begin{array}{c}
  |qx_{i}-p_{i}|< \psi_{i}(q), \, \, 1 \leq i \leq n, \\
  \quad \text{for i.m } \, (\bp,q) \in \mathbb{Z}^{n} \times \mathbb{N}
  \end{array} \right\},
  \end{equation*}
where $\Psi: \mathbb{N} \rightarrow \mathbb{R}_{+}^{n}$, with $\Psi(q)=(\psi_{1}(q), \dots, \psi_{n}(q))$ and $\psi_{i}: \mathbb{N} \rightarrow \mathbb{R}_{+}$ for each $1 \leq i \leq n$. There are many measure theoretic results that have been established for $\w_{n}(\Psi)$. A Dirichlet style theorem for weighted simultaneous approximation can be deduced from Minkowski's linear forms Theorem (see Section 1.4.1 of \cite{BRV16}) to show that $\w_{n}(\bt)= \mathbb{R}^{n}$ for all vectors $\bt=(\tau_{1}, \dots , \tau_{n}) \in \mathbb{R}^{n}_{+}$ satisfying $\sum_{i=1}^{n} \tau_{i}  \leq 1$. Further, there is also a Khintchine style result \cite{K26} that states for $\sum_{i=1}^{n} \tau_{i} > 1$ the set has zero Lebesgue measure. The Hausdorff dimension result for $\w_{n}(\Psi)$ was determined by Rynne \cite{R98} and reads as follows.
 \begin{theorem} \label{rynne}
Let $\bt=(\tau_{1}, \dots, \tau_{n}) \in \R^{n}_{+}$ with $\tau_{1} \geq \dots \geq \tau_{n}$ and suppose that $\sum_{i=1}^{n} \tau_{i} \geq 1$. Then
\begin{equation*}
\dim \w_{n}(\bt) = \underset{1 \leq j \leq n}{\min}
\left\{ \frac{n+1+ \sum_{i=j}^{n}(\tau_{j} - \tau_{i})}{1+\tau_{j}}\right\}.
\end{equation*}
\end{theorem}
It is worth noting that in \cite{R98} Rynne proved more than the above theorem. He obtained a dimension result for $\w_{n}(\bt)$ when the set of points were approximated by integers $q$ lying in any infinite subset $Q \subseteq \N$. \par
  In Rynne's paper \cite{R98}, the $\bt$ approximation function can be replaced with a monotonic decreasing $\Psi$ approximation function providing that the limit
  \begin{equation*}
  \tau^{*}_{i} = \lim _{q \rightarrow \infty} \frac{-\log \psi_{i}(q)}{\log q}
  \end{equation*}
  exists for each $1 \leq i \leq n$. We can do this by using the inequality
  \begin{equation*}
   q^{-\tau^{*}_{i} -\epsilon} \leq \psi_{i}(q) \leq q^{-\tau^{*}_{i}+\epsilon},
   \end{equation*}
   which is valid for sufficiently large $q \in \mathbb{N}$. Observe that, for $\bt^{*}=(\tau^{*}_{1}, \dots, \tau^{*}_{n})$, the corresponding sets then satisfy
   \begin{equation*}
  \w_{n}( \bt^{*}+ \mathbf{\epsilon}) \subset \w_{n}(\mathbf{\Psi}) \subset \w_{n}(\bt^{*}-\mathbf{\epsilon}),
   \end{equation*}
   where $\mathbf{\epsilon}=(\epsilon, \dots, \epsilon) \in \mathbb{R}^{n}$. Thus,
   \begin{equation*}
   \dim \w_{n}(\bt^{*}- \mathbf{\epsilon}) \leq  \dim \w_{n}(\Psi) \leq  \dim \w_{n}(\bt^{*}+ \mathbf{\epsilon}),
   \end{equation*}
   and as $\mathbf{\epsilon}$ is arbitrary we have that $\dim \w_{n}(\Psi)=\dim \w_{n}(\bt^{*})$. \par 
   The goal of this paper is to obtain a general lower bound on the dimension for weighted simultaneous  Diophantine approximation on manifolds. When considering manifolds we look at them locally on some open subset $\U \subset \R^{d}$ and use the following Monge parametrisation without loss of generality
\begin{equation*}
\m:= \left\{ ( \bx, f(\bx) ) : \bx \in \U  \right\} \subseteq \R^{n},
\end{equation*}
where $d$ is the dimension of the manifold, and $f$ is a map such that $f: \U \rightarrow \R^{m}$ with $m$ being the codimension of the manifold such that $d+m=n$. As the manifold is of this form we can consider the approximation of the coordinates $\bx$ and $f(\bx)$ separately. When doing this we will refer to $\bx$ as the independent variables, and the codomain of $f$ as the dependent variables. When considering the set of simultaneously approximable points on manifolds the approximation functions on both the independent and dependent variables are the same. \par 
Much progress has been made in establishing measure theoretic results for the set $\s_{n}(\psi) \cap \m$, we highlight some of these results below. The first research in this field involved extremal manifolds. Spindzhuk established many of the foundational results in this area which he referred to as Diophantine approximation on dependent variables \cite{S67}. A differentiable manifold is called extremal if almost all points, with respect to the induced Lebesgue measure of the manifold, are extremal, whereby we mean that the Dirichlet approximation of the space cannot be improved for almost all points. It was first conjectured \cite{S80} and later proven by Kleinbock and Margulis \cite{KM98} that any non-degenerate submanifold of $\mathbb{R}^{n}$ is extremal. One of the first advances with respect to the Hausdorff dimension of the set $\s_{2}(\psi) \cap \m$ was by Beresnevich, Dickinson, and Velani in \cite{BDV07}, where they determined the dimension of the set of simultaneously approximable points on sufficiently curved planar curves in $\mathbb{R}^{2}$. There is also a related paper \cite{BV07} which uses a similar technique to find the Hausdorff dimension of $\w_{2}(\bt) \cap \m$ for $\bt=(\tau_{1}, \tau_{2})$ bounded below and above by $0$ and $1$ respectively. Both papers give an equality for the dimension rather than just a lower bound as presented in this paper. The following is Theorem 4 from \cite{BV07}. We denote the set of $n$ times continuously differentiable functions by $C^{(n)}$.
   \begin{theorem}[Beresnevich et al. \cite{BV07}] \label{first}
   Let $f$ be a $C^{(3)}$ function over an interval $I_{0}\subset \R$, and let $\mathcal{C}_{f} := \{ (x,f(x)) : x \in I_{0} \}$. Let $\bt=(\tau_{1},\tau_{2})$, where $\tau_{1}$ and $\tau_{2}$ are positive numbers such that $0 < \min \{ \tau_{1}, \tau_{2} \} < 1$ and $\tau_{1} +\tau_{2} \geq 1$. Assume that 
   \begin{equation} \label{ineqdim}
   \dim \left\{ x \in I_{0} : f''(x)=0 \right\} \leq \frac{2-\min \{ \tau_{1}, \tau_{2} \}}{1 + \max \{ \tau_{1}, \tau_{2} \}}.
   \end{equation}
   Then
   \begin{equation*}
   \dim \w_{2}(\bt) \cap \mathcal{C}_{f} =  \frac{2-\min \{ \tau_{1}, \tau_{2} \}}{1 + \max \{ \tau_{1}, \tau_{2} \}}.
   \end{equation*}
   \end{theorem}
   
     Theorem 4 from \cite{BDV07} is the simultaneous case where $\tau_{1}=\tau_{2}$. The common approach in both papers is ubiquity, as established in \cite{BDV06}, to determine the lower bound. The upper bound is found through a combination of Huxley's estimate \cite{H96}, which gives an upper estimate on the number of rational points within a specified neighbourhood of the curve, and the property given by \eqref{ineqdim}. This result has been further improved by Beresnevich and Zorin \cite{BZ10} who showed that the dimension result holds for weakly non-degenerate curves (see Theorem 4 of \cite{BZ10}). In the $n$-dimensional setting Beresnevich et al. \cite{BLVV17} proved the following result.
     \begin{theorem}[Beresnevich et al. \cite{BLVV17}] \label{second} 
     Let $\mathcal{M}$ be any twice continuously differentiable submanifold of $\mathbb{R}^{n}$ of codimension $m$ and let 
     \begin{equation*}
     \frac{1}{n} \leq \tau <\frac{1}{m}.
     \end{equation*} 
     Then
     \begin{equation*}
     \dim \s_{n}(\tau) \cap \mathcal{M} \geq s := \frac{n+1}{\tau+1} -m.
     \end{equation*}
     Furthermore,
     \begin{equation*}
     \ha^{s} ( \s_{n}(\tau) \cap \mathcal{M})= \ha^{s}(\mathcal{M}).
     \end{equation*}
     \end{theorem}
     \begin{remark}
     In the special case where the submanifold $\mathcal{M}$ is a curve this result has been proven for a wider range of $\tau$. In particular, for any analytic non-degenerate curve $\mathcal{C} \subset \R^{n}$, if
     \begin{equation*}
     \frac{1}{n} < \tau < \frac{3}{2n-1},
     \end{equation*}
      then the dimension result of Theorem ~\ref{second} still holds, that is
      \begin{equation*}
       \dim \s_{n}(\tau) \cap \mathcal{C} \geq s := \frac{n+1}{\tau+1} -(n-1).
      \end{equation*}
      See Theorem 7.2 of \cite{B12} for more details. 
      \end{remark}
     A key result required in the proof of Theorem ~\ref{second} is the Mass Transference Principle \cite{BV06}(see Section 2 for more details). Recently, Beresnevich et.al. \cite{BVVZ17} worked on finding an upper bound on the distribution of rational points within a $\psi$-neighbourhood of manifolds. Using this result they proved the following theorem, giving a corresponding upper bound to Theorem ~\ref{second}.
\begin{theorem} \label{haus_upper_bound}
Let $\M_{f}\subset \R^{n}$ be a manifold defined on an open subset $\U \subset \R^{d}$, and suppose that
\begin{equation} \label{require2}
\ha^{s} \left( \left\{ \alpha \in \U: \left| \det \left( \frac{\partial^{2}f_{j}}{\partial \alpha_{1} \partial \alpha_{i}} (\alpha) \right)_{1 \leq i,j \leq m} \right|=0 \right\} \right)=0,
\end{equation}
for 
\begin{equation*}
s=\frac{n+1}{\tau+1}-m.
\end{equation*}
If
\begin{equation*}
d>\frac{n+1}{2}, \quad \text{and} \quad \frac{1}{n} \leq \tau \leq \frac{1}{2m+1},
\end{equation*}
Then
\begin{equation*}
\dim \s_{n}(\tau) \cap \M_{f} \leq s.
\end{equation*}
\end{theorem}
This is a simplified version of the main result established in \cite{BVVZ17}. In particular the  convergent Hausdorff measure result was proven for general functions $\psi(q) \geq q^{-1/(2m+1)}(\log q)^{2/(2m+1)}$. Further still, the result was proven for the general case of inhomogeneous simultaneous approximation. \par 
Since the establishment of Theorem ~\ref{haus_upper_bound} there has been several results that allow for a broader range of manifolds. Recently Simmons relaxed condition \eqref{require2} as follows (see Theorem 2.1 of \cite{S18}). Suppose there exists some $k \in \N$ such that 
\begin{equation*}
s\big( 1+ \frac{k}{2m+k} \big) >n+1,
\end{equation*}
and
\begin{equation*}
rank \left( \by. \frac{\partial^{2} f}{\partial \alpha_{i} \partial \alpha_{j} }(\alpha) \right)_{1 \leq i,j \leq d} \geq k, \quad \forall \by \in \R^{m} \backslash \{0\},
\end{equation*}
for almost all $\alpha \in \U$ (w.r.t the Hausdorff $(s-m)$-measure). Then
\begin{equation*}
\dim \s_{n}(\tau) \cap \mathcal{M}_{f} \leq s-m.
\end{equation*}
\par 
Results of this type have also been proven for hypersurfaces. In \cite{H20} Huang proved an upper bound on the number of rational points within a small neighbourhood of a general $C^{(l)}$ hypersurface $\mathcal{H} \subseteq \R^{n}$ with Gaussian curvature bounded above zero. This theorem provided a variety of results, including the following (Theorem 5 of \cite{H20}).
\begin{theorem}[Huang \cite{H20}] \label{huang}
Let $n \geq 3$ be an integer and let
\begin{equation*}
l=\max \left\{ \left\lfloor \frac{n-1}{2} \right\rfloor +5, n+1 \right\}.
\end{equation*}
 For any approximation function $\psi$, any $s>\frac{n-1}{2}$, and any $C^{(l)}$ hypersurface $\mathcal{H} \subseteq \R^{n}$ with nonvanishing Gaussian curvature everywhere except possibly on a set of zero Hausdorff $s$-measure, we have that
\begin{equation*}
\ha^{s}(\s_{n}(\psi) \cap \mathcal{H})=0 \quad \text{if} \quad \sum_{q=1}^{\infty} \psi(q)^{s+1}q^{n-1-s}< \infty.
\end{equation*}
\end{theorem}
In the special case of $\tau$-approximable functions the above theorem implies that
\begin{equation*}
\dim \s_{n}(\tau) \cap \mathcal{H} \leq \frac{n+1}{\tau+1}-1.
\end{equation*}
Thus we note that when considered against the applicable range of approximation functions from Theorem ~\ref{second} this is the best possible upper bound. \par 
Similar results have also been found for simultaneous approximation on affine subspaces. For a matrix $M \in \R^{d \times (n-d)}$, and row vector $\alpha \in \R^{n-d}$ let 
\begin{equation*}
\mathfrak{A}:= \left\{ \left( \bx, (1, \bx).\left( \begin{array}{c}
\alpha \\
M
\end{array} \right) \right): \bx \in \R^{d} \right\}.
\end{equation*}
 Huang and Liu  \cite{HL18} proved that such affine subspace $\mathfrak{A} \subseteq \R^{n}$ with bounded Diophantine properties on the matrix $\left( \begin{array}{c}
\alpha \\
M
\end{array} \right) \in \R^{(d+1) \times (n-d)}$, and any approximation function $q^{-\tau}$ with $\tau \geq 1/n$, then
\begin{equation*}
\dim \s_{n}(\tau) \cap \mathfrak{A} \leq \frac{n+1}{\tau+1}-(n-d).
\end{equation*}
\par 
 For general manifolds Theorem \ref{second} and Theorem \ref{haus_upper_bound} collectively give the following corollary.
\begin{corollary} \label{simultaneous_dim_manifold}
Let $\M_{f}$ be a manifold satisfying \eqref{require2} and $d>\frac{n+1}{2}$. Suppose that
\begin{equation*}
\frac{1}{n} \leq \tau \leq \frac{1}{2m+1},
\end{equation*}
 then
 \begin{equation*}
 \dim \s_{n}(\tau) \cap \M_{f} = \frac{n+1}{\tau+1}-m.
 \end{equation*}
 \end{corollary}
\par 
 In this paper we adapt the arguments given in \cite{BLVV17} for the proof of Theorem \ref{second} to establish the following result, a weighted simultaneous version of Theorem ~\ref{second}.
\begin{theorem} \label{final}
Let $\m := \left\{ (\bx, f(\bx)) : \bx \in \U \subset \R^{d} \right\}$ where $f: \U \rightarrow \R^{m}$ with $ f \in C^{(2)}$. Let $\bt = (\tau_{1}, \dots, \tau_{n}) \in \R^{n}_{>0}$ with
\begin{equation*}
\tau_{1} \geq \dots \geq \tau_{d} \geq \max_{d+1 \leq i \leq n} \left\{ \tau_{i}, \frac{1-\sum_{j=1}^{m}\tau_{j+d}}{d} \right\}, \quad \text{and} \quad \sum_{i=1}^{m}\tau_{d+i} < 1.
\end{equation*}
Then
\begin{equation*}
\dim \left( \w_{n}(\bt) \cap \m \right) \geq \underset{1 \leq j \leq d}{\min} \left\{ \frac{n+1 + \sum_{i=j}^{n}(\tau_{j}-\tau_{i})}{\tau_{j}+1} - m \right\}.
\end{equation*}
\end{theorem}
\begin{remark}
\begin{enumerate}[label=\arabic*)]
\item Note that the minimum is taken over only the $\tau_{i}$ for $1 \leq i \leq d$, that is the approximation functions over the independent variables $\bx \in \R^{d}$. This condition may only be due to the setup of our proof and the fact that all approximation functions over the independent variables are larger than all the dependent variable approximation functions. While we suspect this to be unnecessary, the mass transference style result used in the proof of Theorem ~\ref{final} forces such conditions to be applied. \\
\item We only have a lower bound here rather than equality. This lower bound agrees with both Theorem ~\ref{first} and Theorem ~\ref{huang} so in these cases this is the best lower bound. In order to prove the upper bound result to Theorem ~\ref{final} we need an upper bound result on the number of rational points within a $\Psi$-neighbourhood of a manifold. While there are many various results for the simultaneous case (e.g. \cite{H96}, \cite{BVVZ17},\cite{H20}), a weighted simultaneous version is yet to be proven.
\end{enumerate}
\end{remark}
We would like to generalise Theorem ~\ref{final} to more general approximation functions. To achieve this we must apply some constraints on our approximation function. Given a decreasing approximation function $\Psi=(\psi_{1}, \dots , \psi_{n})$ define the upper order $v(\Psi)=\left(v_{1}, \dots , v_{n} \right)$ of $\Psi$ at infinity by 
\begin{equation} \label{limsup}
v_{i}:= \underset{q \rightarrow \infty}{\limsup} \frac{-\log \psi_{i}(q)}{\log q} \, , \quad  1 \leq i \leq n.
\end{equation}
Given such a function, we can state the following Corollary.
\begin{corollary} \label{general}
Let $\m := \left\{ (\bx, f(\bx)) : \bx \in \U \subset \R^{d} \right\}$ where $ f: \U \rightarrow \R^{m}$ with $f \in C^{(2)}$. For any approximation function $\Psi = (\psi_{1}, \dots , \psi_{n})$ such that \eqref{limsup} are positive finite, and
\begin{equation*}
v_{1} \geq v_{2} \geq \dots \geq v_{d} \geq \max_{d+1 \leq i \leq n} \left\{ v_{i}, \frac{1-\sum_{i=d+1}^{n}v_{i}}{d} \right\}, \quad \text{and} \quad \sum_{i=d+1}^{n}v_{i} < 1.
\end{equation*}
we have that
\begin{equation*}
\dim \left( \w_{n} ( \Psi) \cap \m \right) \geq  \underset{1 \leq j \leq d}{\min} \left\{ \frac{n+1+ \sum_{i=j}^{n} \left( v_{j}-v_{i}\right)}{v_{j} + 1 }\right\}.
\end{equation*}
\end{corollary}
\begin{proof}
By  properties of the approximation function given by \eqref{limsup} we have that, for any $\epsilon >0$ there exists a $ q_{0} \in \N$ such that for all $q>q_{0}$
\begin{equation*}
\psi_{i}(q) \geq q^{-v_{i}-\epsilon} \, , \text{for each } \, \, 1 \leq i \leq n.
\end{equation*}
Using this property, for $\mathbf{\epsilon}=(\epsilon, \dots, \epsilon) \in \R^{n}_{+}$ we obtain that
\begin{equation*}
\w_{n}( v(\Psi)+ \mathbf{\epsilon})  \subset \w_{n}(\Psi) 
\end{equation*}
so by Theorem ~\ref{final}, and letting $\epsilon \rightarrow 0$, we have that
\begin{align*}
\dim \left( \w_{n}(\Psi) \cap \m \right) & \geq \lim_{ \epsilon \to 0^{+}}\dim \left( \w_{n}(v(\Psi)+ \epsilon) \cap \m \right), \\
& \geq \underset{1 \leq i \leq d}{\min} \left\{ \frac{n+1+ \sum_{i=j}^{n} \left( v_{j}-v_{i} \right)}{ v_{j}+1} \right\},
\end{align*}
as required.
\end{proof}
\begin{remark}
Note that this proof is similar to the proof of Corollary 1 from \cite{R98}. However we can use the weaker condition of the $\limsup$ rather than $\lim$ as we only need the lower bound rather than equality.
\end{remark}
The remainder of this paper is laid out as follows. In the following section we recall some key theorems required in the proof of Theorem ~\ref{final}. The main result of this section is the Mass Transference Principle from balls to rectangles, proven in \cite{WWX15}. In Section 3 we prove a Dirichlet style result on weighted simultaneous approximation over manifolds. This result is vital in order to apply the mass transference style theorem of Section 2. In the final section we combine the results of Sections 2 and 3 to prove Theorem ~\ref{final}.

\section{Mass Transference theorems}

 The Mass Transference Principle (MTP) was developed by Beresnevich and Velani in \cite{BV06}. Before stating the MTP we need to introduce some notation. Let $B=B(x,r)$ denote a ball $B \subset\R^{n}$, with centre $x \in \R^{n}$ and radius $r \in \R_{+}$. Denote $B^{s}$ to be the ball with centre $x$ and radius $r^{s/n}$, that is 
\begin{equation*} 
 B^{s}=B(x, r^{s/n}).
\end{equation*}
In particular, note that $B^{n}=B$.
\begin{theorem}[The Mass Transference Principle] \label{thm:mtp} Let $\mathcal{U}$ be an open subset of $\mathbb{R}^{k}$. Let $\{B_{i} \}_{i \in \mathbb{N}}$ be a sequence of balls in $\mathbb{R}^{k}$ centred in $\mathcal{U}$ with radius $r_{i} \rightarrow 0 $ as $i \rightarrow \infty$. Let $s>0$ and suppose that for any ball $B$ in $\mathcal{U}$
\begin{equation*}
\mathcal{H}^{k} \left( B \cap \limsup_{i \rightarrow \infty} B_{i}^{s} \right) = \mathcal{H}^{k}(B).
\end{equation*}
Then, for any ball $B$ in $\mathcal{U}$
\begin{equation*}
\mathcal{H}^{s} \left( B \cap \limsup_{i \rightarrow \infty} B_{i}^{k} \right) = \mathcal{H}^{s}(B).
\end{equation*}
\end{theorem}    
\begin{remark}
In this theorem we take the dimension function $f: \mathbb{R}_{+} \rightarrow \mathbb{R}_{+}$ to be $f(r)=r^{s}$. However in \cite{BV06} it is shown that we can choose any dimension function providing it is a continuous increasing function with $f(r) \rightarrow 0$ as $r \rightarrow 0$
\end{remark}
The  MTP is a key tool in establishing Theorem ~\ref{second}. The general method for establishing such results is to formulate a Dirichlet style theorem in the desired setting, and then increasing the radius by the power of $s/n$ to form the original $\limsup$ set. Geometrically we can consider this as moving from a $\limsup$ set of balls with radius $r$ to a limsup set of balls with radius $r^{s/n}$. With the set $\w_{n}(\bt)$ we have the problem that the set is described by a $\limsup$ set of hyperrectangles rather than balls, so we cannot use the MTP. The following theorem by Wang, Wu and Xu \cite{WWX15} is similar to the MTP with the difference that we move from a $\limsup$ set of balls to a $\limsup$ set of hyperrectangles. In this case we use very similar notation to that of above, with the exception that we introduce a weight vector. Define a weight vector to be of the form $\ba=(a_{1}, a_{2}, \dots, a_{n}) \in \R^{n}_{+}$ with $a_{1} \geq \dots \geq a_{n}>0$. Given a ball $B=B(\bx,r) \subset \R^{n}$, define $B^{\ba}=B\left(\bx,(r^{a_{1}},r^{a_{2}}, \dots , r^{a_{n}})\right)= \left\{ y \in \R^{n} : |x_{i}-y_{i}|< r^{a_{i}} \right\}$, i.e. a hyperrectangle with side lengths $2r^{a_{i}}$ and centre $\bx$.
\begin{theorem}[Wang, Wu, Xu \cite{WWX15}] \label{wang}
Let $(\bx_{j})_{j \in \N}$ be a sequence of points in $[0,1]^{n}$ and $(r_{j})_{j \in \N}$ be a sequence of positive real numbers such that $r_{j} \rightarrow 0$ as $j \rightarrow \infty$, then let $B_{j}=B(\bx_{j},r_{j})$. Let $\ba=(a_{1}, a_{2}, \dots , a_{n}) \in \R^{n}_{+}$ such that $ a_{1} \geq a_{2} \geq \dots \geq a_{n} \geq 1$. Suppose that 
\begin{equation*}
\mu_{n}\left( \underset{j \rightarrow \infty}{\limsup} B_{j} \right) = 1,
\end{equation*}
where $\mu_{n}$ is the $n$ dimensional Lebesgue measure. Then,
\begin{equation*}
\dim \left(  \underset{j \rightarrow \infty}{\limsup} B_{j}^{\ba}  \right) \geq \underset{1 \leq k \leq n}{\min} \left\{ \frac{n+ \sum_{i=k}^{n}(a_{k}-a_{i})}{a_{k}} \right\}.
\end{equation*}
\end{theorem}
We see that to use Theorem ~\ref{wang} we need two key ingredients. Firstly we need a $\limsup$ set of balls with full Lebesgue measure. Secondly, we need to construct a weight vector $\ba$ that we can use to transform our set of full Lebesgue measure to our desired $\limsup$ set of hyperrectangles.
We illustrate how this theorem can be used to attain the lower bound of Theorem ~\ref{rynne}, as shown in \cite{D17}.
\begin{proof} (Lower bound of Theorem ~\ref{rynne}) \\
Let 
\begin{equation*}
N:= \left\{ (\bp,q) \in \Z^{n} \times \N : 0 \leq p_{i} \leq q \, \text{ for}\, 1 \leq i \leq n \right\},
\end{equation*}
for $\bp=(p_{1}, \dots,p_{n})$, and also let
\begin{equation*}
B_{(\bp,q)}=B\left( \frac{\bp}{q}, q^{-1-1/n} \right)
\end{equation*}
for some $(\bp,q) \in N$. By Dirichlet's theorem we have that
\begin{equation*}
\mu_{n}\left( \underset{(\bp,q) \in N}{\limsup} B_{(\bp,q)} \right)=1.
\end{equation*}
If we take $\ba=(a_{1}, \dots, a_{n})$ to be the weight vector with coefficients
\begin{equation*}
a_{i}=\frac{n(1+\tau_{i})}{1+n} \quad \text{for} \quad 1 \leq i \leq n,
\end{equation*}
then
\begin{equation*}
B_{(\bp, q)}^{\ba}= \left\{ \bx \in \R^{n} : \left| x_{i}-\frac{p_{i}}{q} \right| < q^{-1-\tau_{i}} \quad \text{for} \quad 1 \leq i \leq n \right\}.
\end{equation*}
So we have that 
\begin{equation*}
\dim \w_{n}(\bt) = \dim \left( \underset{(\bp,q) \in N}{\limsup} B_{(\bp,q)}^{\ba} \right).
\end{equation*}
Using Theorem \ref{wang} we get that
\begin{align*}
\dim \w_{n}(\bt) \geq& \underset{1 \leq j \leq n}{\min} \left\{ \frac{ n+ \sum_{i=j}^{n} \left( \frac{n(1+\tau_{j})}{1+n} -\frac{n(1+\tau_{i})}{1+n} \right)}{ \frac{n(1+\tau_{j})}{1+n}} \right\}, \\
\geq &  \underset{1 \leq j \leq n}{\min} \left\{ \frac{ n+1+ \sum_{i=j}^{n} \left( \tau_{j} -\tau_{i} \right)}{1+\tau_{j}} \right\},
\end{align*}
as required.
\end{proof}
The last measure theoretic result we will be using to prove Theorem ~\ref{final} is a lemma from \cite{BV08}, which essentially states that the Lebesgue measure of a $\limsup$ set remains the same when the balls are altered by some fixed constant.
\begin{lemma}
\label{lemma:balls}
Let $\{B_{i}\}$ be a sequence of balls in $\mathbb{R}^{k}$ with $\mu_{k}(B_{i}) \rightarrow 0$ as $i \rightarrow \infty$. Let $\{U_{i}\}$ be a sequence of Lebesgue measurable sets such that $U_{i} \subset B_{i}$ for all $i$. Assume that for some $c>0$, $\mu_{k}(U_{i}) \geq c\mu_{k}(B_{i})$ for all $i$. Then the sets 
\begin{equation*}
\lim_{i \rightarrow \infty} \sup U_{i} \qquad \text{and} \qquad \lim_{i \rightarrow \infty } \sup B_{i}
\end{equation*}
have the same Lebesgue measure.
\end{lemma}
We can use Lemma ~\ref{lemma:balls} to change the radius of the balls used in our construction of the $\limsup$ set by a constant and still ensure we have full Lebesgue measure. 

\section{Dirichlet Style theorem on Manifolds}

In order to apply Theorem ~\ref{wang} we construct a $\limsup$ set of balls with full Lebesgue measure. We achieve this by varying the approximation functions only over the dependent variables, so we can form a $\limsup$ set from the balls centred at certain rational points in the independent variable space. The theorem below constructs such a set.
\begin{theorem} \label{diristyle}
Let $\m:=\left\{ ( \bx, f(\bx)) : \bx \in \U \subset \R^{d} \right\}$ where $f: \U \rightarrow \R^{m}$ with $f \in C^{(2)}$. Let $ \bt=(\tau_{1}, \dots , \tau_{m}) \in \R^{m}_{>0}$, and let $\tb=\frac{1}{m}\sum_{i=1}^{m}\tau_{i}$. If
\begin{equation*} 
 \tb m <1,
\end{equation*}
 then for any $\bx \in \U$ there is an integer $Q_{0}$ such that for any $Q \geq Q_{0}$ there exists $(p_{1}, \dots, p_{n}, q) \in \Z^{n} \times \N$ with $1 \leq q \leq Q$ and $(\frac{p_{1}}{q}, \dots, \frac{p_{d}}{q}) \in \U$ such that
\begin{equation}\label{independent}
\left| x_{i} - \frac{p_{i}}{q} \right| < \frac{4^{m/d}}{q(Q^{1-\tb m})^{1/d}}, \quad 1 \leq i \leq d,
\end{equation}
and 
\begin{equation} \label{dependent}
\left| f_{j} \left( \frac{p_{1}}{q}, \dots , \frac{p_{d}}{q} \right) - \frac{p_{d+j}}{q} \right| < \frac{q^{-\tau_{j}-1}}{2}. \quad 1 \leq j \leq m.
\end{equation}
Further, for any $\bx \in \U \backslash \Q^{d}$ there exists infinitely many $(p_{1}, \dots, p_{n}, q) \in \Z^{n} \times \N$ with $(\frac{p_{1}}{q}, \dots, \frac{p_{d}}{q}) \in \U$ satisfying \eqref{dependent} and
\begin{equation} \label{end}
\left| x_{i} - \frac{p_{i}}{q} \right| < 4^{m/d} q^{-1-(1-\tb m)/d}, \quad 1 \leq i \leq d.
\end{equation}
\end{theorem} 
Before proving Theorem ~\ref{diristyle} we will state several properties of our manifold $\m$ that we will be using. Given that $\m$ is constructed by a twice continuously differentiable function $f$ we can choose a suitable $\U$ such that, without loss of generality, the following two constants exist:
\begin{equation} \label{c}
C = \underset{\underset{1 \leq j \leq m}{1 \leq i,k \leq d}}{ \max} \, \underset{\bx \in \U}{\sup} \left| \frac{\partial^{2} f_{j}}{\partial x_{i} \partial x_{k} } (\bx) \right| < \infty,
\end{equation}
and 
\begin{equation} \label{d}
D= \underset{\underset{1 \leq j \leq m}{1 \leq i \leq d}}{\max} \, \underset{\bx \in \U}{\sup} \left| \frac{\partial f_{j}}{\partial x_{i} } (\bx) \right| < \infty.
\end{equation}
A brief outline of the proof is as follows; firstly we alter the system of inequalities to a suitable form so Minkowski's Theorem for systems of linear forms can be applied. We then use Taylor's approximation Theorem to return the system of inequalities to the initial form and show that the dependent variable inequalities can be displayed in terms of the independent variable approximation. We finish by concluding that there are infinitely many different integer solutions via a proof by contradiction. The proof given below is a generalisation of the proof of Theorem 4 in \cite{BLVV17} to the case of approximations with weights.
\begin{proof}
Define
\begin{equation*}
g_{j} := f_{j} - \sum_{i=1}^{d} x_{i} \frac{\partial f_{j}}{\partial x_{i}}, \quad 1 \leq j \leq m,
\end{equation*}
and consider the system of inequalities
\begin{align}
\left| q g_{j}(\bx) + \sum_{i=1}^{d} p_{i} \frac{\partial f_{j}}{\partial x_{i}}(\bx) - p_{d+j} \right| &< \frac{Q^{-\tau_{i}}}{4}, \quad  \, 1 \leq j \leq m, \label{eq1} \\
\left| q x_{i}-p_{i} \right| &< \frac{4^{m/d}}{Q^{(1-\tb m)/d}}, \quad \, 1 \leq i \leq d, \label{eq2} \\
|q| & \leq Q. \label{eq3}
\end{align}
Taking the product of the right hand side of the above inequalities, and taking the determinant of the matrix
\begin{equation*}
A=
\left( \begin{array}{ccccccc}
g_{1} & \frac{\partial f_{1}}{\partial x_{1}} & \dots & \frac{\partial f_{1}}{\partial x_{d}} & -1 & \dots  & 0 \\
\vdots & \vdots & & \vdots & \vdots & \ddots & \vdots \\
g_{2} & \frac{\partial f_{m}}{\partial x_{1}} & \dots & \frac{\partial f_{m}}{\partial x_{d}} & 0 & \dots  & -1 \\
x_{1} & -1 & \dots & 0 & 0 & \dots & 0 \\
\vdots & \vdots& \ddots & \vdots& \vdots & \ddots& \vdots \\
x_{d} & 0 & \dots & -1 & 0 & \dots & 0 \\
1 & 0 & \dots& \dots & \dots & \dots & 0 
\end{array} \right),
\end{equation*}
then by Minkowski's Theorem for systems of linear forms, there exists a non-zero integer solution $(\bp,q) \in \Z^{n+1}$ to the inequalities \eqref{eq1}-\eqref{eq3}. We now show that this system of inequalities implies inequalities \eqref{independent}-\eqref{dependent}. Firstly, fix some $\bx \in \U$ and as $\U$ is open there exists a ball $B( \bx, r)$ for some $r>0$ which is contained in $\U$. Define
\begin{equation*}
\q := \left\{ Q \in \N : \left( 4^{-m} Q^{1-\tb m}\right)^{-1/d} < \min \left\{ 1 , r, \left( \frac{1}{2 Cd^{2}} \right)^{1/2} \right\} \right\},
\end{equation*}
where $C$ is defined by \eqref{c}. As $\tb m< 1$ we have that $\left( 4^{-m} Q^{1-\tb m}\right)^{-1/d} \to 0$ as $Q \to \infty$, so there exists an integer $Q_{0}$ such that for all $Q \geq Q_{0}$ we have that $Q \in \q$. We will show that for any $Q \in \q$ the solution $(p_{1}, \dots, p_{n}, q)$ to the system of inequalities \eqref{eq1}-\eqref{eq3} is a solution to \eqref{independent}-\eqref{dependent}. \par 
Suppose $q=0$. By the definition of the set $\q$ we have that
\begin{equation*}
\left( 4^{-m} Q^{1-\tb m}\right)^{-1/d} < 1.
\end{equation*}
By the set of inequalities \eqref{eq2} we have that $|p_{i}|<1$, hence $p_{i}=0$ for all $1 \leq i \leq d$. Further, from \eqref{eq1} we can see that
\begin{equation*}
|p_{d+j}|< \frac{Q^{-\tau_{i}}}{4} < 1,
\end{equation*}
for $1\leq j \leq m$. This would conclude that our solution $(p_{1}, \dots, p_{n}, q)=\boldsymbol{0}$ which contradicts Minkowski's Theorem for systems of linear forms, thus $|q| \geq 1$. Without loss of generality we will assume that $q \geq 1$. Dividing \eqref{eq2} by $q$ gives us that $\left(\frac{p_{1}}{q}, \dots, \frac{p_{d}}{q}\right) \in B(\bx, r) \subset \U$, and note that \eqref{independent} is satisfied upon dividing \eqref{eq2} by $q$.\par 
 Lastly we need to prove that \eqref{eq1}-\eqref{eq3} implies \eqref{dependent}. By Taylor's approximation Theorem
 \begin{equation*}
 f_{j}\left(\frac{p_{1}}{q}, \dots, \frac{p_{d}}{q} \right)= f_{j}(\bx) + \sum_{i=1}^{d} \frac{\partial f_{j}}{\partial x_{i}}(\bx) \left( \frac{p_{i}}{q}-x_{i} \right) + R_{j}( \bx, \hat{\bx} ),
 \end{equation*}
 for some $\hat{\bx}$ on the line connecting $\bx$ and $\left(\frac{p_{1}}{q}, \dots, \frac{p_{d}}{q}\right)$, and
 \begin{equation*}
 R_{j}(\bx, \hat{\bx}) = \frac{1}{2} \sum_{i=1}^{d} \sum_{k=1}^{d} \frac{\partial^{2}f_{j}}{\partial x_{i} \partial x_{k}} (\hat{\bx}) \left( \frac{p_{i}}{q} - x_{i} \right) \left( \frac{p_{k}}{q} - x_{k} \right).
 \end{equation*}
 We may rewrite \eqref{eq1} using Taylor's theorem and our definition of $g_{j}$ as
 \begin{equation*}
 \left| q g_{j}(\bx) + \sum_{i=1}^{d} p_{i} \frac{\partial f_{j}}{\partial x_{i}}(\bx) - p_{d+j} \right|= \left| q f_{j}\left(\frac{p_{1}}{q}, \dots, \frac{p_{d}}{q} \right) - p_{d+j} - q R_{j}(\bx, \hat{\bx}) \right|.
 \end{equation*}
 Using the triangle inequality and the assumption that 
 \begin{equation}
 |q R_{j}(\bx, \hat{\bx})| < \frac{q^{-\tau_{i}}}{4}, \label{taylorerror}
 \end{equation}
 we obtain that
 \begin{equation*}
 \left| q f_{j}\left( \frac{p_{1}}{q}, \dots, \frac{p_{d}}{q} \right) - p_{d+j} \right| < \frac{Q^{-\tau_{i}}}{4} + \frac{q^{-\tau_{i}}}{4}.
 \end{equation*}
Noting the monotonicity of the approximation function and dividing by $q$ we obtain
 \begin{equation*}
 \left|  f_{j}\left( \frac{p_{1}}{q}, \dots, \frac{p_{d}}{q} \right) - \frac{p_{d+j}}{q} \right| < \frac{q^{-\tau_{i}-1}}{2},
 \end{equation*}
 thus \eqref{dependent} is satisfied. To complete the first part of the theorem it remains to show that \eqref{taylorerror} is satisfied for all $Q \in \q$. Using the definition of $R_{j}(\bx, \hat{\bx})$ we have that
 \begin{align*}
 |qR_{j}(\bx, \hat{\bx})| =& \left| \frac{q}{2} \sum_{i=1}^{d} \sum_{k=1}^{d} \frac{\partial^{2}f_{j}}{\partial x_{i} \partial x_{k}} (\hat{\bx}) \left( \frac{p_{i}}{q} - x_{i} \right) \left( \frac{p_{k}}{q} - x_{k} \right) \right|, \\
 \leq & \frac{Cqd^{2}}{2}\left( \frac{4^{m/d}}{q Q^{(1-\tb m)/d}}\right)^{2},
 \end{align*}
 for $1 \leq j \leq m$. Hence we must show that
 \begin{equation*}
  \frac{Cqd^{2}}{2}\left( \frac{4^{m/d}}{q Q^{(1-\tb m)/d}}\right)^{2} < \frac{q^{-\tau_{i}}}{4},
  \end{equation*}
  for $1 \leq j \leq m$. Rearranging the equation we obtain the inequality
  \begin{equation*}
  \left( \frac{4^{m}}{Q^{(1-\tb m)}}\right)^{1/d} < \left( \frac{q^{1-\tau_{i}}}{2 Cd^{2}} \right)^{1/2}.
  \end{equation*}
  Considering that $\tb m < 1$ we have that each $\tau_{i}<1$, and so $\underset{q \in \N}{\inf}q^{1-\tau_{i}}=1$ for each $1 \leq i \leq m$. Thus by the definition of the set $\q$ the above inequality is satisfied by all $Q \in \q$, so \eqref{taylorerror} is true for all $1 \leq j \leq m$. \par
 We now prove the second part of the theorem, that is that there is infinitely many integer vector solutions. Suppose that there are only finitely many such $q$ and let $A$ be the corresponding set. As $\bx \in \U \backslash \Q^{d}$ there exists some $1 \leq j \leq d$ where $x_{j} \notin \Q$. Fix such $j$, then there exists some $\delta >0$ such that
 \begin{equation*}
 \delta \leq \underset{q \in A , \, p_{j} \in \Z}{\min} |qx_{j}- p_{j}|.
 \end{equation*}
 By \eqref{independent} we now have that
    \begin{equation*}
    \delta \leq |qx_{j} - p_{j}| \leq \frac{4^{m/d}}{Q^{(1-\tb m)/d}}.
    \end{equation*}
However, as $\q$ is an infinite set and $Q^{(1-\tb m)/d} \rightarrow \infty$ as $Q \rightarrow \infty$, we have a contradiction so there are infinitely many different $q$. Lastly, as $q\leq Q$, we can replace $Q$ by $q$ in \eqref{independent} to obtain \eqref{end} as desired.
\end{proof}

\section{Proof of theorem \ref{final}}

We are now in a position to prove Theorem ~\ref{final}. To do so we construct a $\limsup$ set of balls satisfying the conditions of Theorem ~\ref{diristyle}. The $\limsup$ set will thus have full Lebesgue measure. Next we choose a suitable weight vector $\ba$ that we use to transform our $\limsup$ set of balls to a $\limsup$ set of hyperrectangles with a known lower bound for its Hausdorff dimension. The proof is completed by showing that the constructed $\limsup$ set is at least contained within our set $\w_{n}(\bt) \cap \m$, thus our lower bound is a lower bound for $\w_{n}(\bt) \cap \m$.

\begin{proof}
Take the set
\begin{multline}
N(f, \tau) := \Bigg\{ (\bp, q) \in \Z^{n} \times \N : \left( \frac{p_{1}}{q}, \dots , \frac{p_{d}}{q} \right) \in \U, \, \\ \text{and} \, \left| f_{j}\left( \frac{p_{1}}{q}, \dots, \frac{p_{d}}{q} \right) - \frac{p_{d+j}}{q} \right| < q^{-\tau_{d+j}-1}, \, \, 1 \leq j \leq m \Bigg\} . \nonumber
\end{multline}
In view of Theorem ~\ref{diristyle} we have that for almost all $\bx \in \U$ there are infinitely many different vectors $(\bp, q ) \in N(f, \tau)$ satisfying
\begin{equation*}
\left| x_{i}-\frac{p_{i}}{q}\right| < 4^{m/d} q^{-1-(1-\tb m)/d}  , \quad 1 \leq i \leq d,
\end{equation*}
where $\tb = \frac{1}{m} \sum_{i=1}^{m}\tau_{d+i}$. By Lemma ~\ref{lemma:balls}, we can choose a constant $k>0$ such that for almost every $\bx \in \U$ there are infinitely many different vectors $(\bp, q ) \in N(f, \tau)$ satisfying
\begin{equation*}
\left| x_{i}- \frac{p_{i}}{q}\right| < k q^{-1-(1-\tb m)/d} , \quad 1 \leq i \leq d.
\end{equation*}
Take the ball 
\begin{equation*}
B_{(\bp,q)} := \left\{ \bx \in \U : \left| x_{i}-\frac{p_{i}}{q}\right| < k q^{-1-(1-\tb m )/d} \quad \text{for} \quad 1 \leq i \leq d \right\}.
\end{equation*}
By Theorem ~\ref{diristyle} and Lemma ~\ref{lemma:balls} we have that
\begin{equation*}
\mu_{d}\left( \underset{(\bp,q) \in N(f, \tau)}{\limsup} B_{(\bp,q)} \right)=1,
\end{equation*}
where $\mu_{d}$ is the $d$ dimensional Lebesgue measure. Let $\ba=(a_{1}, \dots, a_{d}) \in \R^{d}_{+}$ be a weight vector with each
\begin{equation}
a_{i}=\frac{d(1+ \tau_{i})}{d + 1-\tb m}, \, \quad 1 \leq i \leq d. \label{ai}
\end{equation} 
Note that by the condition that $\tau_{i} \geq \frac{1-\tb m}{d}$ for all $1 \leq i \leq d$, we have that each $a_{i}\geq1$. $B_{(\bp,q)}^{\ba}$ is the hyperrectangle with the following properties: 
\begin{equation*}
B_{(\bp,q)}^{\ba}= \left\{ \bx \in \U: \left| x_{i}-\frac{p_{i}}{q} \right| < k^{a_{i}} q^{-1-\tau_{i}}, \, \, 1 \leq i \leq d \right\}.
\end{equation*}
By Theorem ~\ref{wang} we have that
\begin{equation*}
\dim \left( \underset{(\bp,q) \in N(f , \tau)}{\limsup} B_{(\bp,q)}^{\ba} \right) \geq \underset{1 \leq j \leq d}{\min} \left\{ \frac{d + \sum_{i=j}^{d}(a_{j}- a_{i})}{a_{j}} \right\}.
\end{equation*}
Replacing each $a_{i}$ with \eqref{ai} we have that 
\begin{align*}
 \dim & \left( \underset{(\bp,q) \in N(f , \tau)}{\limsup} B_{(\bp,q)}^{\ba} \right) \geq   \underset{1 \leq j \leq d}{\min}\left\{ \frac{d + \sum_{i=j}^{d}\left(\frac{d(1+ \tau_{j})}{d + 1-\tb m}- \frac{d(1+ \tau_{i})}{d + 1-\tb m}\right)}{\frac{d(1+ \tau_{j})}{d + 1-\tb m}} \right\}, \\
& \quad \quad \quad \quad = \underset{1 \leq j \leq d}{\min}\left\{ \frac{d(d + 1-\tb m)+ \sum_{i=j}^{d}(d(1+ \tau_{j})- d(1+ \tau_{i}))}{d(1+ \tau_{j})} \right\}, \\
& \quad = \underset{1 \leq j \leq d}{\min}\left\{ \frac{d+1-\tb m+ \sum_{i=j}^{d}(\tau_{j}- \tau_{i})}{(1+ \tau_{j})} \right\}.
 \end{align*}
 Using the definition of $\tb$ and that $d=n-m$ we may rewrite this as
 \begin{align*}
  \dim & \left( \underset{(\bp,q) \in N(f , \tau)}{\limsup} B_{(\bp,q)}^{\ba} \right) \geq \underset{1 \leq j \leq d}{\min} \left\{ \frac{n-m + 1-\sum_{i=d+1}^{n}\tau_{i} + \sum_{i=j}^{d} (\tau_{j}- \tau_{i})}{1+ \tau_{j}} \right\}, \\
 & \quad \quad \quad  =  \underset{1 \leq j \leq d}{\min} \left\{ \frac{n-m + 1+\sum_{i=d+1}^{n}(\tau_{j}-\tau_{i})-m\tau_{j} + \sum_{i=j}^{d} (\tau_{j}- \tau_{i})}{1+ \tau_{j}} \right\}, \\
  & \quad \quad =  \underset{1 \leq j \leq d}{\min} \left\{ \frac{n + 1 -m(1+\tau_{j})+ \sum_{i=j}^{n} (\tau_{j}- \tau_{i})}{1+ \tau_{j}} \right\}, \\
  & \quad =  \underset{1 \leq j \leq d}{\min} \left\{ \frac{n+1 + \sum_{i=j}^{n} (\tau_{j}- \tau_{i})}{1+ \tau_{j}}-m \right\},
 \end{align*}
 as required. We now finish by showing that
 \begin{equation*}
\dim(\w_{n}(\bt) \cap \m) \geq \dim\left(\underset{(\bp,q) \in N(f , \tau)}{\limsup} B_{(\bp,q)}^{\ba}\right).
\end{equation*}
 Note that any $\bf{y} \in$ $\w_{n}(\bt) \cap \m$ must have infinitely many solutions $(\bp,q) \in \mathbb{Z}^{n} \times \mathbb{N}$ to the following system of inequalities
\begin{align}
| qx_{i}-p_{i}| &< q^{-\tau_{i}}, \, 1 \leq i \leq d,  \label{inequal1}\\
|qf(\bx)-p_{d+j}| & < q^{-\tau_{d+j}}, \, 1 \leq j \leq m, \label{inequal2}
\end{align}
where $\bf{y}=(\bx, f(\bx))$ for some $\bx=(x_{1}, \dots, x_{d}) \in \U$. Let the set of $\bx$ satisfying \eqref{inequal1}-\eqref{inequal2} be denoted by $\pi_{d}(\w_{n}(\bt) \cap \m)$. This set is the orthogonal projection of $\w_{n}(\bt) \cap \m$ onto $\R^{d}$. A result of fractal geometry states that a bi-Lipschitz mapping of a set has the same Hausdorff dimension of the original set (see Proposition 3.3 of \cite{F04}). As the projection $\pi_{d}$ is bi-Lipschitz it is sufficient to prove that 
\begin{equation*}
\dim \pi_{d}(\w_{n}(\bt) \cap \m) \geq\underset{1 \leq j \leq d}{\min} \left\{ \frac{n+1 + \sum_{i=j}^{n} (\tau_{j}- \tau_{i})}{1+ \tau_{j}}-m \right\}. \\ \label{dimension}
\end{equation*} 
Let $\bx \in B_{(\bp,q)}^{\ba}$ for some $(\bp,q) \in N(f, \tau)$. On using the triangle inequality, the mean-value theorem, and \eqref{d} we have that for any $1 \leq j \leq m$,
\begin{align*}
\left| f_{j}(\bx)-\frac{p_{d+j}}{q} \right| \leq & \left| f_{j}(\bx)-f_{j}\left(\frac{p_{1}}{q}, \dots, \frac{p_{d}}{q}\right) \right| + \left| f_{j}\left(\frac{p_{1}}{q}, \dots, \frac{p_{d}}{q}\right)-\frac{p_{d+j}}{q} \right|, \\
< & \left| \left(\frac{\partial f_{1}}{\partial x_{1}}, \dots, \frac{\partial f_{d}}{\partial x_{d}}\right). \left( \bx- \left(\frac{p_{1}}{q}, \dots, \frac{p_{d}}{q}\right) \right) \right| + \frac{q^{-1-\tau_{d+j}}}{2}, \\
< & D \sum_{i=1}^{d} \left| x_{i}-\frac{p_{i}}{q}\right| + \frac{q^{-1-\tau_{d+j}}}{2}, \\
< & Dd \underset{1 \leq i \leq d}{\max}\left| x_{i}-\frac{p_{i}}{q}\right| + \frac{q^{-1-\tau_{d+j}}}{2}, \\
< & Dd k^{a_{d}}q^{-1-\tau_{d}} + \frac{q^{-1-\tau_{d+j}}}{2}.
\end{align*}
 We can choose $k$ sufficiently small, and note that $\tau_{d} \geq \max_{1 \leq j \leq m} \tau_{d+j}$, so that we have 
\begin{equation*}
\left| f_{j}(\bx)-\frac{p_{d+j}}{q} \right| <  q^{-1-\tau_{j}}, \quad 1 \leq j \leq m.
\end{equation*}
We have that
\begin{equation*}
B_{(\bp,q)}^{\ba} \subseteq \left\{ \bx \in \U : \begin{array}{c}
 \left| x_{i}-\frac{p_{i}}{q} \right| < q^{-1-\tau_{i}}, \quad 1 \leq i \leq d, \quad \\
 \text{for i.m } \, (\bp, q) \in N(f, \tau) \subset \Z^{n} \times \N
 \end{array} \right\}.
\end{equation*}
Hence for any $\bx \in \underset{(\bp,q) \in N(f , \tau)}{\limsup} B_{(\bp,q)}^{\ba}$,  \eqref{inequal1}-\eqref{inequal2} are satisfied for infinitely many $(\bp,q) \in \mathbb{Z}^{n} \times \mathbb{N}$, thus
\begin{equation*}
\dim \pi_{d}(\s_{n}(\bt) \cap \m) \geq \dim \left(\underset{(\bp,q) \in N(f , \tau)}{\limsup} B_{(\bp,q)}^{\ba}\right),
\end{equation*}
as required.
\end{proof}

\section{Conclusion}

Using the arguments above and, principally applying the MTP of \cite{BV07}, we have established a lower bound for $\w_{n}(\bt) \cap \M$ which coincides with that of $\s_{n}(\psi) \cap \M$ from Theorem ~\ref{second}. The natural question is can equality be determined. That is, can an upper bound be found which agrees with our calculated lower bound? Thus achieving a complete analogue of Theorem~\ref{final}. In trying to attain an upper bound, it is likely necessary to find an estimate for the number of rational points within a $\bt$-neighbourhood of the manifold. There are a variety of results on the cardinality of rational points within a simultaneous $\psi$-neighbourhood of curves, manifolds, hypersurfaces, and affine subspaces (see \cite{H96}, \cite{BVVZ17}, \cite{H20}, \cite{HL18} respectively). Unfortunately, no such results have been found for the number of rational points within a weighted $\Psi$-neighbourhood of such subsets. It may be possible to adapt the proofs of the simultaneous results to give us weighted version of such results, but this is yet to be proven. We suspect such result would lead to suitable upper bound corresponding to Theorem ~\ref{final}. \par 
We remark that since the our proof of Theorem~\ref{final} there has been developments in Mass Transference style theorems. In particular, Wang and Wu \cite{WW19} have developed a Mass Transference Principle from rectangles to rectangles(MTPRR). One condition of using the MTPRR is that the $\limsup$ set being transformed is a ubiquitous system of rectangles (see Definition 2 of \cite{WW19} for more details). Unfortunately our Dirichlet-style theorem, Theorem~\ref{diristyle}, is not sufficient to prove our set of points are ubiquitous. We suspect that while using the MTPRR would improve the conditions on the $\bt$ approximations, we may need to apply extra conditions to our manifold in order to achieve ubiquity. We intend to pursue this in a further paper. \par 
Lastly, note that Corollary ~\ref{general} whilst being relatively general does not cover all approximation functions. We provide no results for functions with infinite upper order (see \eqref{limsup}), and also provide imprecise lower dimension results for approximation functions with different upper and lower orders. For example, an approximation function defined by a step function bounded between two functions $q^{- \tau_{1}}$ and $q^{-\tau_{2}}$ would have different upper and lower bounds. It would be of interest to extend the class of approximating functions somehow.

\bibliographystyle{plain}
\bibliography{bibref}
\end{document}